\documentclass[11pt,a4paper]{article}
\usepackage{amsmath,amssymb,amsthm,amsfonts,latexsym,graphicx,subfigure}

\usepackage[affil-it]{authblk}

\usepackage{overpic, rotating, graphicx}

\usepackage[numbers,sort&compress]{natbib}

\usepackage[a4paper,margin=3cm]{geometry}

\usepackage{comment}
\usepackage{color}
\usepackage{tikz}
\usetikzlibrary{topaths,calc}
\usetikzlibrary{positioning}
\tikzset{bluenode/.style={circle,fill=gray!50,minimum size=0.4cm,inner sep=0pt},}
\tikzset{rednode/.style={circle,fill=black!100,minimum size=0.4cm,inner sep=0pt},}

\DeclareMathOperator{\id}{Id}

\newcommand{\R}{\ensuremath{\mathbb{R}}}
\newcommand{\N}{\ensuremath{\mathbb{N}}}

\usepackage{hyperref}

\theoremstyle{plain}
\newtheorem{theorem}{Theorem}[section]

\newtheorem{proposition}[theorem]{Proposition}
\newtheorem{corollary}[theorem]{Corollary}

\theoremstyle{definition}

\newtheorem{example}[theorem]{Example}
\theoremstyle{remark}
\newtheorem{remark}{Remark}[section]

\newcommand{\bel}[1]{\begin{equation}\label{#1}}

\newcommand{\be}{\begin{equation}}

\newcommand{\ba}{\begin{eqnarray}}
\newcommand{\ea}{\end{eqnarray}}
\newcommand{\rf}[1]{(\ref{#1})}

\newcommand{\qe}{\end{equation}}

\begin{document}
	\bibliographystyle{plain} 
	\title{Random walks and Laplacians on hypergraphs:\\ When do they match?}
	
	\author[1,2,3]{Raffaella Mulas} 
	\author[4,5]{Christian Kuehn}
	\author[4]{Tobias Böhle}
	\author[3,6]{Jürgen Jost}
	\affil[1]{The Alan Turing Institute, The British Library, London NW1 2DB, UK}
	\affil[2]{University of Southampton, University Rd, Southampton SO17 1BJ, UK}
	\affil[3]{Max Planck Institute for Mathematics in the Sciences, Inselstr.~22, 04103 Leipzig, Germany}
	\affil[4]{Department of Mathematics, Technical University of Munich, Boltzmannstr.~3, 85748 Garching b.~M\"unchen, Germany}
	\affil[5]{Complexity Science Hub Vienna, Josefst\"adter Str.~39, 1080 Vienna, Austria}
	\affil[6]{Santa Fe Institute for the Sciences of Complexity, 1399 Hyde Park Road Santa Fe, New Mexico 87501, USA}
	\date{}
	\maketitle
	\allowdisplaybreaks[4]

	\begin{abstract}
		We develop a general theory of random walks on hypergraphs which includes, as special cases, the different models that are found in literature. In particular, we introduce and analyze general random walk Laplacians for hypergraphs, and we compare them to hypergraph normalized Laplacians that are not necessarily related to random walks, but which are motivated by biological and chemical networks. We show that, although these two classes of Laplacians coincide in the case of graphs, they appear to have important conceptual differences in the general case. We study the spectral properties of both classes, as well as their applications to Coupled Hypergraph Maps: discrete-time dynamical systems that generalize the well-known Coupled Map Lattices on graphs. Our results also show why for some hypergraph Laplacian variants one expects more classical results from (weighted) graphs to generalize directly, while these results must fail for other hypergraph Laplacians.
	\end{abstract}
	
	\section{Introduction}
	Random walks on graphs were  introduced in the seminal paper
	\cite{Courant28}. Since then, the theory of random walks on graphs has been
	developed by many people, and profound connections to other fields have been
	established, see for instance \cite{Grimmett10}.
	
	Recently, hypergraphs have received much attention, constituting a more general framework  than graphs that is more adequate for the modelling of many empirical interaction data \cite{bitcoin,Theis,epidemic,social,battiston2020networks,bradde09,Sun}. Therefore, it seems natural to also investigate random processes on hypergraphs. In \cite{rwhyp,Banerjee,carlettidynamical,carletti2021random}, different models for random walks on hypergraphs have been introduced and analyzed. Here, we present a systematic account of random walks on hypergraphs, by developing it from general principles. We study, in particular, general random walk Laplacians for hypergraphs, and we compare them to other known types of hypergraph Laplacians, which are motivated by biological and chemical networks. We find that none of these classes of Laplacians is to be preferred over the other, as they encode complementary information. We analyze and compare the spectral properties of the different Laplacians, and we investigate their applications to discrete-time dynamical systems. We remark that in this paper, for simplicity and in line with most of the literature, we do not consider lazy random walks, although it is not overly difficult to extend our considerations to those. That is, we assume that a random walker is not allowed to remain at a vertex, but is required to move to one the neighbors during the next time step.
	
	The paper is structured as follows. In Section \ref{rgr}, we recall the main terminology concerning random walks on both simple graphs and weighted, directed graphs. We discuss, in particular, the theoretical details of the corresponding \emph{random walk Laplacian}. For a graph $\Gamma$ on $N$ nodes, this is an operator that can be written, in matrix form, as $L=\id-D^{-1}A$, where $\id$ is the $N\times N$ identity matrix, $A$ is the adjacency matrix and $D$ is the diagonal degree matrix of the graph. Hence, if $v_1,\ldots,v_N$ denote the vertices of $\Gamma$, then for $i\neq j$,
	\begin{equation*}
	-L_{ij}=\frac{A_{ij}}{d(v_i)}
	\end{equation*}
	is the probability that a random walker goes from $v_i$ to $v_j$. It is therefore not surprising that the \emph{spectrum} of $L$, i.e., the multiset of its eigenvalues, is strictly related to the properties of this random walk. But additionally, it is also well-known that this spectrum encodes qualitative and structural properties of the graph that do not necessarily have to do with its corresponding random walk. This is why the study of the spectral properties of $L$ is widely present in graph theory, and it also finds many applications, as for instance in dynamical systems and data analysis~\cite{Chung}.
	
	In Section \ref{rhyp}, we discuss the conceptual differences between the various random walks on hypergraphs that generalize the classical ones on graphs. We suggest, in particular, that a hypergraph random walk is preferable if it takes the hyperedge sizes (therefore, not only the pairwise connections between vertices) into account. We then develop, in Section \ref{section:Lrw}, a general theory of hypergraph random walks, which includes the examples from Section \ref{rhyp} as special cases. In particular, for any given random walk on a hypergraph, we consider the random walk Laplacian $\mathcal{L}$ such that
	\begin{equation*}
	\mathcal{L}_{ij}=\begin{cases}1 & \text{ if }i=j\\
	-\mathbb{P}(v_i\rightarrow v_j) & \text{ if }i\neq j,
	\end{cases}\end{equation*}
	so that its entries encode, as in the case of graphs, the probabilities of the random process. We show that, since the probabilities are \emph{pairwise} relations between vertices, every hypergraph random walk Laplacian $\mathcal{L}$ coincides with the random walk Laplacian of a (possibly weighted and directed) graph. As a consequence, we show how the spectral properties of such Laplacians can be inferred from the results that are known for graphs. Moreover, Section \ref{section:orientedL} is devoted to the comparison between the hypergraph random walk Laplacians and the normalized Laplacians for oriented hypergraphs that were motivated by chemical interactions and introduced in \cite{JM}; for investigations into properties of these hypergraph Laplacians see e.g.~\cite{MulasZhang,Mulas-Cheeger,JM21b,Master-Stability,BKJM}. We observe that, while these latter operators do not encode properties of random walks in general, their spectrum is nevertheless capable of encoding qualitative properties of the hypergraph that are not encoded by the spectra of the random walk Laplacians. Thus, these two classes of hypergraph Laplacians are complementary to each other. 
	
	Finally, in Section \ref{section:CHM} we study mathematical foundations of Coupled Hypergraph Maps (CHMs): dynamical systems on hypergraphs that were introduced in \cite{BKJM} as a generalization of Coupled Map Lattices (CMLs)~\cite{Kaneko}. CHMs are based on Laplacian-type coupling, therefore the generated dynamics changes when a different type of Laplace operator is considered. There are actually two main motivations, why it is crucial to study the impact of different hypergraph operators. From a theoretical perspective, there are many mathematical methods one wants to generalize to CHMs such as the renormalization group~\cite{LemaitreChate}, which require precise algebraic statements about invariance and commutation properties of the hypergraph coupling operator. From a more application-oriented perspective, there is an ongoing debate, under which conditions one can reduce, or cannot reduce, dynamical systems on hypergraphs to the case of graphs~\cite{battiston2020networks,KuehnBick,NeuhaeuserMellorLambiotte,SkardalArenas}. Our results for CHMs clearly show that already on an algebraic level of invariance and commutation properties, the precise choice of hypergraph Laplacian is crucial, whether or not one can expect to generalize results about dynamics on graphs to hypergraphs. 
	
	\section{Random walks on graphs}
	\label{rgr}
	Since the ultimate interest of \cite{Courant28} was in the discretization of
	partial differential equations, they considered only regular graphs. On regular
	graphs, in
	a terminology to be explained in a moment, the forward and the backward
	operator coincide. Of course, since then, the general case has been treated in
	the literature. We follow here the simple presentation in \cite{biology}.\newline
	All graphs will be
	connected. In particular, we exclude isolated vertices, as they cannot
	participate in random processes in a nontrivial way. We start with an
	undirected and unweighted graph $\Gamma$ with vertex set $V$ and write $w\sim v$ and call  the vertices
	$v,w$ neighbors, if they are connected by an edge. The degree $d(v)$ of $v$ is
	the number of its neighbors.\newline
	The rule for a random walker is simple. When at time $n\in N$ at vertex $v$,
	she chooses one of the neighbors of $v$ with equal probability as her position
	at time $n+1$. Thus, let $p(n,w,v)$ be the probability that when starting at
	$v$ at time $0$, she reaches $w$ at time $n$. Then we have the forward equation 
	\bel{g1}
	p(n+1,w,v)=\sum_{z\sim w} \frac{1}{d(z)}p(n,z,v)
	\qe
	and the backward equation 
	\bel{g2}
	p(n+1,w,v)=\frac{1}{d(v)}\sum_{u\sim v} p(n,w,u).
	\qe
	Eq.\ \rf{g1} says that whatever reaches $w$ after $n+1$ steps from $v$ must have reached one
	of its neighbors after $n$ steps. Eq.\ \rf{g2} says that whatever reaches $w$
	after $n+1$ steps from $v$ must reach $w$ after $n$ steps from one of the
	neighbors of $v$. We therefore have a forward process
	\bel{g3}
	p(n+1,w,v)-p(n,w,v)=\sum_{z\sim w} \frac{1}{d(z)}p(n,z,v)-p(n,w,v) =:-L^\ast p(n,w,v)
	\qe
	and a backward process
	\bel{g4}
	p(n+1,w,v)-p(n,w,v)=\frac{1}{d(v)}\sum_{u\sim v}p(n,w,u)-p(n,w,v) =:-L p(n,w,v),
	\qe
	where the operator
	\bel{g5}
	Lf(v)=f(v)-\frac{1}{d(v)}\sum_{u\sim v} f(u)
	\qe
	operates on the variable $v$, and
	\bel{g6}
	L^\ast g(w)=g(w)-\sum_{z\sim w}\frac{1}{d(z)}g(z)
	\qe
	operates on $w$.
	We write $L^\ast$ here because the two operators are adjoints with respect to the
	scalar product
	\begin{equation*}
	\langle f,g\rangle_V :=\sum_v f(v)g(v)
	\end{equation*}
	for functions on $V$. We note, however, that $L$ is self-adjoint with respect to the
	product
	\begin{equation*}
	(f,g)_V:= \sum_v d(v) f(v)g(v). 
	\end{equation*}
	Thus, $L$ has a real spectrum, and the two operators $L^\ast$ and $L$ are isospectral.\newline
	This can be easily generalized to weighted graphs. If $\omega_{vw}=\omega_{wv}\ge 0$ is the
	weight of the edge between $v$ and $w$ (we may put $\omega_{vw}=0$ if $w
	\nsim v$), the probability of a random walker moving from $v$ to $w$ is
	proportional to the weight $\omega_{vw}$. Putting
	\begin{equation*}
	d(v):=\sum_w \omega_{vw},
	\end{equation*}
	the generators then remain as in \rf{g5}, \rf{g6}.
	We may also allow for self-loops, that is, for the possibility that
	$\omega_{vv}>0$. In that case, the walker stays at $v$ with probability
	$\frac{\omega_{vv}}{d(v)}$, and moves to some neighboring vertex with probability
	$\frac{d(v)-\omega_{vv}}{d(v)}$ only. One speaks of a lazy random walker in
	that situation. \newline
	So far, we have assumed that the graph is undirected, that is,
	$\omega_{vw}=\omega_{wv}$. Of course, a random walk is still possible on a
	directed graph, that is, when that symmetry need not hold. When $
	\omega_{vw}$ is the weight of the edge from $v$ to $w$, we put
	\begin{equation*}
	d_{out}(v):=\sum_w \omega_{vw},\quad  d_{in}(v):=\sum_w \omega_{wv}.
	\end{equation*}
	The generators then are
	\bel{g11}
	Lf(v)=f(v)-\frac{1}{d_{out}(v)}\sum_{u\sim v} f(u)
	\qe
	and
	\bel{g12}
	L^\ast g(w)=g(w)-\sum_{z\sim w}\frac{1}{d_{in}(z)}g(z).
	\qe
	In this case, however, $L$ needs no longer be self-adjoint, and it may have
	imaginary eigenvalues.\newline
	When we assume
	\begin{equation*}
	\sum_v p(0,v,v)=1\quad \text{and}\quad p(0,w,v)=0 \quad \text{for } w\neq v,
	\end{equation*}   then
	\bel{g15}
	\sum_w p(n,w,v)=1=\sum_v p(n,w,v)\quad \text{for all }v,w\in V, n\in \N,
	\qe
	that is, the total probability for the location and for the origin of the random walker remains
	$1$ for the processes \rf{g3}, \rf{g4}  governed by $L,L^\ast$.\newline
	The processes need not converge to constant limits. This is already seen
	for the simplest graph, consisting of two vertices $v_1,v_2$ connected by an
	edge. If $p(0,v_1,v_1)=1$, that is, the random walker starts at $v_1$, then $p(n,v_1,v_1)=1$, $p(n,v_2,v_1)=0$ for even
	$n$ and $p(n,v_1,v_1)=0$, $p(n,v_2,v_1)=1$ for odd $n$, that is, the random
	walker is always at $v_1$ at even times and at $v_2$ at odd times. More generally, 
	an oscillatory pattern occurs if and only if $\Gamma$ is bipartite. Otherwise, the
	process does converge to a constant equilibrium.
	
	\begin{remark}We may also consider a random walk with an absorbing boundary. For that purpose, we select some non-empty boundary set $V_0 \subsetneq V$ for which $V\backslash V_0$ is connected. A random walker starting at some $v\in V\backslash V_0$ then moves according to the above rule until she reaches some $v_0\in V_0$ where she then remains. When $V\backslash V_0$ is finite, the random walker will eventually end up at some boundary vertex. 
	\end{remark}
	\begin{remark}
		In graph theory, often the unnormalized Laplacian
		\begin{equation*}
		\Lambda f(v)=d(v) f(v)-\sum_{u\sim v} f(u)
		\end{equation*}
		is considered. That operator, however, does not allow for a probabilistic
		interpretation because a process governed by it in general does not satisfy \rf{g15}.
	\end{remark}
	Besides the random walk, there is another natural process that can be defined
	on graphs, the \emph{diffusion}. Here, we have some substance
	with concentration $f(0,v)$ at time $0$. On a graph, the rule is simply that
	whatever is at $v$ at time $n$ is distributed among its neighbors $w$ at time
	$n+1$ (in proportion to the weights $\omega_{vw}$ when the graph is weighted
	or directed), and likewise what accumulates at $w$ at time $n+1$ is received
	from its neighbors. Therefore, the resulting  process $f(n,w,v)$ follows the
	evolution equations \rf{g1}, \rf{g2}. This is completely analogous to the
	relation between Brownian motion and heat conduction in the continuous case in,
	say, Euclidean spaces.  
	
	\section{Random walks on hypergraphs}\label{rhyp}
	Let $\Gamma=(V,E)$ be a hypergraph with vertex set $V$ and hyperedge set $E$, again
	unweighted and undirected to start with. Now,
	there are various possibilities for defining a random walk. In the simplest
	case, a  walker that is
	at $v$ has two options for choosing her next position:
	\begin{enumerate}
		\item[1)] Go to a vertex $w$ that is connected to $v$ with a probability that depends on the number of hyperedges that $v$ and $w$ have in common, but not on the sizes of such hyperedges.
		\item[2)] First choose among the hyperedges containing $v$ with equal
		probability, and when one such hyperedge is selected, go to any of its
		vertices other than $v$  with equal probability. 
	\end{enumerate}
	The second choice seems preferable as the first one does not take the hyperedge sizes, therefore the
	hypergraph structure, into account. The process would be the same if we
	replaced any hyperedge by a complete graph, that
	is, replace the hypergraph by a multi-graph with the same connectivity
	pattern.\newline
	This second choice also originates from a different model. Replace the hypergraph by a bipartite graph, with one class $A$ of vertices corresponding to those of the hypergraphs, and the other class $B$ corresponding to the hyperedges. And when the hypergraph is directed, that resulting graph would also be directed. In that graph, we then consider a two-step random process. Start at a node from $A$, go with equal probability to one of the vertices in $B$ to which it has a (directed) connection, and then from such a vertex in $B$, seek one of the vertices in $A$ to which it has a (directed) connection with equal probability. In the undirected case, there then is the possibility that the random walker returns to its starting vertex in $A$, which would correspond to a slight modification of the hypergraph random walk, making it a lazy random walk. In the directed case, however, the two processes are equivalent. The point is that in this process, two probabilistic choices are made, one for a hyperedge, and the other for a vertex in that hyperedge, and that is translated into the two-step random walk in the bipartite graph. \newline
	The directed case just discussed also covers the situation of oriented hypergraphs. 
	In fact, as argued in \cite{JM}, we should include more structure in order
	to have a good theory of Laplace type operators, generalizing \rf{g5},
	\rf{g6},  on hypergraphs. The vertex set of each hyperedge $e$ should be
	grouped into two disjoint (and not necessarily non-empty) classes $V_{in}(e)$ and
	$V_{out}(e)$, like the reactants and products of a chemical reaction that is
	represented by the hyperedge $e$. In that case,  rule 2) would become
	\begin{enumerate}
		\item[3)] If a hyperedge with $v\in V_{in}(e)$ is chosen, take one of the
		vertices in $V_{out}(e)$ with equal probability, and conversely. That is,
		a random walker is not allowed to move from an input (output) to another
		input (output) of a hyperedge. 
	\end{enumerate}
	\begin{remark}
		When the vertices of a hyperedge are classified into in- and outputs, we
		may call that hyperedge \emph{oriented}. 
		Changing the orientation then means exchanging in- and outputs. The concept of an \emph{oriented}  hypergraph is to be distinguished from that of a \emph{directed} one, because in  such an oriented hypergraph, for each oriented hyperedge, we find one \emph{orientation} but both \emph{directions}. 

	\end{remark}
	\begin{remark}
		On  a \emph{graph}, an orientation of an edge is fixed if we declare one of
		the vertices as in- and the other as output. Changing the orientation of an edge does not change a random walk or the Laplacian of the graph. 
	\end{remark}
	
	\begin{remark}
		Diffusion on hypergraphs can be defined as in the case of graphs. The analogy between random walks and diffusion processes pertains to hypergraphs. 
	\end{remark}
	We now make an important \emph{observation}. A random walk on a hypergraph as conceived here is defined in terms of transition probabilities between vertices, that is, by probabilities for going from one vertex to another one. But this is a \emph{pairwise} relation between vertices, and any pairwise relation between vertices can be encoded by an ordinary graph. Therefore, also the various random walks defined here all possess equivalent representations in terms of weighted graphs, as will be formally stated and proved in Theorem \ref{thm:effective}. Therefore, in particular, the generators of such random walks, the Laplace operators to be studied in the next section, have to satisfy all the properties that normalized graph Laplacians enjoy, like the maximum principle. The most important such properties will be listed in Corollary \ref{cor:connected} below.
	
	\section{Hypergraph random walk Laplacians}\label{section:Lrw}
	We now develop a \emph{general} theory of random walks and random walk Laplacians for hypergraphs. We say that a discrete-time random process is a \emph{random walk on a hypergraph $\Gamma=(V,E)$} if it is a path of vertices $(v_{t_0},v_{t_1},v_{t_2},\ldots)$ such that:
	\begin{itemize}
		\item For each $k\geq 1$, $v_{t_k}$ is chosen with a certain probability that only depends on $v_{t_{k-1}}$, i.e., the process is a Markov chain;
		\item Given any two vertices $v_i\neq v_j$, the probability $\mathbb{P}(v_i\rightarrow v_j)$ of going from $v_i$ to $v_j$ is such that
		\begin{equation*}
		\mathbb{P}(v_i\rightarrow v_j)\neq 0 \iff \exists e\in E: v_i, v_j\in e.
		\end{equation*}
	\end{itemize}
	We say that an $N\times N$ matrix $\mathcal{L}$ is a \emph{random walk Laplacian} if there exist a connected hypergraph $\Gamma=(V,E)$ on $N$ vertices $v_1,\ldots,v_N$ and a random walk on $\Gamma$ such that
	\begin{equation}\label{RWL}
	\mathcal{L}_{ij}=\begin{cases}1 & \text{ if }i=j\\
	-\mathbb{P}(v_i\rightarrow v_j) & \text{ if }i\neq j.
	\end{cases}
	\end{equation}
	Given a random walk Laplacian $\mathcal{L}$ and a corresponding random walk, let $\mathcal{D}$ be an $N\times N$ diagonal matrix with positive diagonal entries, and $\mathcal{A}$ be an $N\times N$ non-negative matrix with zeros on the diagonal, such that, for $i\neq j$,
	\begin{equation}\label{eq:generalRW}
	\mathbb{P}(v_i\rightarrow v_j)=\frac{\mathcal{A}_{ij}}{\mathcal{D}_{ii}}.
	\end{equation}Note that this is always possible if no further assumptions are made. In fact, one can always let $\mathcal{D}=\id$ be the $N\times N$ identity matrix, and let $\mathcal{A}_{ij}=\mathbb{P}(v_i\rightarrow v_j)$. However, if one requires $\mathcal{A}$ to be symmetric, it is not always possible to find $\mathcal{A}$ and $\mathcal{D}$ that satisfy \eqref{eq:generalRW}. Moreover, in general, $\mathcal{A}$ and $\mathcal{D}$ that satisfy \eqref{eq:generalRW} are not uniquely determined.\newline
	Now, let $\mathcal{A}$ and $\mathcal{D}$ that satisfy \eqref{eq:generalRW}, with no further assumptions. For each $i$,
	\begin{equation}\label{eq:AD}
	1=\sum_{j}\mathbb{P}(v_i\rightarrow v_j)=\sum_{j}\frac{\mathcal{A}_{ij}}{\mathcal{D}_{ii}},
	\end{equation}therefore $\mathcal{D}^{-1}\mathcal{A}$ is a $1$-row sum matrix. Moreover, we can write 
	\begin{equation*}
	\mathcal{L}=\id-\mathcal{D}^{-1}\mathcal{A}.
	\end{equation*}
	Its transpose is
	\begin{equation*}
	\mathcal{L}^\top=\id-\mathcal{A}\mathcal{D}^{-1},
	\end{equation*}hence
	\begin{equation*}
	\mathcal{L}^\top_{ij}=\begin{cases}1 & \text{ if }i=j\\
	-\mathbb{P}(v_j\rightarrow v_i) & \text{ if }i\neq j.
	\end{cases}
	\end{equation*}
	Therefore, while $\mathcal{L}$ describes where something
	goes to, as the operators \eqref{g5} and \eqref{g11} in the case of graphs, its transpose $\mathcal{L}^\top$ says where something comes from, as the operators \eqref{g6} and \eqref{g12} in the case of graphs. Also, $\mathcal{L}$ is a zero-row sum matrix,  while $\mathcal{L}^\top$ is a zero-column sum
	matrix. And since they are transpose of each other, $\mathcal{L}$ and $\mathcal{L}^\top$ are isospectral.
	
	\begin{example}If we consider the first random walk process described in Section \ref{rhyp}, then $\mathcal{D}_{ii}:=\sum_{e\in E: v_i\in e}|e|-1$ and
		\begin{equation*}
		\mathcal{A}_{ij}:=\begin{cases}
		0 & \text{ if }i=j\\
		|\{e\in E:v_i,v_j\in e\}| & \text{ otherwise}
		\end{cases}
		\end{equation*}satisfy \eqref{eq:generalRW}. Also, the corresponding random walk Laplacian is the one that has been studied in \cite{rwhyp}.
	\end{example}
	
	\begin{example}\label{exam:second}If we consider the second random walk process described in Section \ref{rhyp}, then $\mathcal{D}_{ii}:=|\{e\in E:v_i\in e\}|$ and 
		
		\begin{equation*}
		\mathcal{A}_{ij}:=\begin{cases}
		0 & \text{ if }i=j\\
		\sum_{e\in E:\,v_i,v_j\in e}\frac{1}{|e|-1} & \text{ otherwise}
		\end{cases}
		\end{equation*}
		satisfy \eqref{eq:generalRW}. In fact, in this case,
		\begin{align*}
		\mathbb{P}(v_i\rightarrow v_j)&=\sum_{e\in E\,:\, v_i,v_j\in e}\mathbb{P}(\text{going from $v_i$ to }e)\cdot \mathbb{P}(\text{going from $e$ to }v_j|v_i\rightarrow e)\\
		&=\sum_{e\in E\,:\, v_i,v_j\in e}\frac{1}{|\{e\in E:v_i\in e\}|}\cdot \frac{1}{|e|-1}\\
		&=\frac{\mathcal{A}_{ij}}{\mathcal{D}_{ii}}.
		\end{align*}
		The corresponding random walk Laplacian is the one that has been studied by Banerjee in \cite{Banerjee}.
	\end{example}
	\begin{example}If we consider the random walk process that has been introduced in \cite{carlettidynamical}, then
		$\mathcal{D}_{ii}:=\sum_{j\neq i}\sum_{e\in E\,:\,v_i,v_j\in e}|e|-1$ and
		\begin{equation*}
		\mathcal{A}_{ij}:=\begin{cases}
		0 & \text{ if }i=j\\
		\sum_{e\in E\,:\,v_i,v_j\in e}|e|-1 & \text{ otherwise}
		\end{cases}
		\end{equation*}satisfy \eqref{eq:generalRW}.
	\end{example}
	
	\begin{remark}
		All three examples above coincide with the classical random walk process on undirected graphs that we described in Section \ref{rgr}. In particular, in all these cases, the corresponding random walk Laplacian coincides with the operator in \eqref{g5}. In the more general case of hypergraphs, the three processes in the examples and the corresponding operators are not necessarily equivalent. 
	\end{remark}
	
	\begin{example}If we consider the third random walk process described in Section \ref{rhyp}, then $\mathcal{D}_{ii}:=|\{e\in E:v_i\in e\}|$ and 
		\begin{equation*}
		\mathcal{A}_{ij}:=\begin{cases}
		0 & \text{ if }i=j\\
		\sum_{\substack{e\in E\,:\,v_i,v_j\in e\\ \text{anti-oriented}}}\frac{1}{|w\in e\,:\,v_i,w \text{ anti-oriented in }e |}&\text{ otherwise}
		\end{cases}
		\end{equation*}
		satisfy \eqref{eq:generalRW}. In fact, in this case,
		
		\begin{align*}
		\mathbb{P}(v_i\rightarrow v_j)&=\sum_{e\in E\,:\, v_i,v_j\in e}\mathbb{P}(\text{going from $v_i$ to }e)\cdot \mathbb{P}_o(\text{going from $e$ to }v_j|v_i\rightarrow e)\\
		&=\sum_{\substack{e\in E\,:\, v_i,v_j\in e \\\text{anti-oriented}}}\frac{1}{|\{e\in E:v_i\in e\}|}\cdot \frac{1}{|w\in e\,:\,v_i,w \text{ anti-oriented in }e |}\\
		&=\frac{\mathcal{A}_{ij}}{\mathcal{D}_{ii}},
		\end{align*}
		where $\mathbb{P}_o$ indicates that the random walker can only go from
		$v_i$ through the hyperedge $e$ to an anti-oriented vertex. Note that, with this construction, $\mathcal{A}$ is not necessarily symmetric.
	\end{example}
	
	Now, given any random walk on a hypergraph, the corresponding Laplacian coincides with the random walk Laplacian of a weighted (and possibly directed) graph, as shown by the next theorem.
	
	\begin{theorem}\label{thm:effective}
		Let $\Gamma=(V,E)$ be a hypergraph. Given a random walk on $\Gamma$, let $\mathcal{L}$ be the corresponding Laplacian and let $\mathcal{A}$ and $\mathcal{D}$ that satisfy \eqref{eq:generalRW}. Let also $\mathcal{G}:=(V,\mathcal{E},\omega)$ be the weighted (and possibly directed) graph which is defined by letting
		\begin{equation*}
		(v_i,v_j)\in  \mathcal{E} \iff \mathcal{A}_{ij}>0
		\end{equation*}and $\omega_{v_iv_j}:=\mathcal{A}_{ij}$. Then, $\mathcal{L}$ coincides with the normalized Laplacian associated with $\mathcal{G}$.\newline
		We say that $\mathcal{G}$ is an effective underlying graph of the random process.
	\end{theorem}
	\begin{proof}
		The degree of a vertex in the weighted graph $\mathcal{G}$ is
		\begin{equation*}
		d(v_i)=\sum_{j} \omega_{v_iv_j}=\sum_{j}\mathcal{A}_{ij}=\mathcal{D}_{ii},
		\end{equation*}
		where the last equality follows by \eqref{eq:AD}. Hence, the degree matrix of $\mathcal{G}$ is $\mathcal{D}$, while the adjacency matrix of $\mathcal{G}$ is $\mathcal{A}$. Therefore, $\mathcal{L}=\id-\mathcal{D}^{-1}\mathcal{A}$ coincides with the normalized Laplacian associated with $\mathcal{G}$.
	\end{proof}
	
	\begin{figure}[h]
		\centering
		\begin{overpic}[width = .9\textwidth]{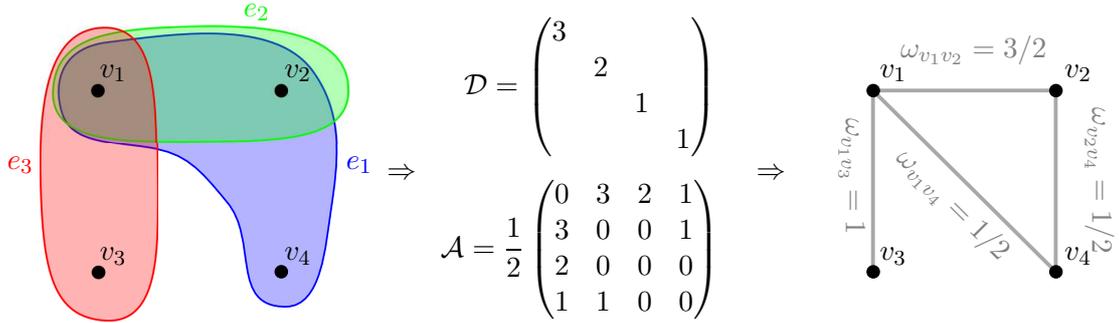}
			\put(30,15){\color{blue}$e_1$}
			\put(20,30){\color{green}$e_2$}
			\put(-3,15){\color{red}$e_3$}
			
			\put(6,24){$v_1$}
			\put(24,24){$v_2$}
			\put(6,6){$v_3$}
			\put(24,6){$v_4$}
			
			\put(82,24){$v_1$}
			\put(100,24){$v_2$}
			\put(82,6){$v_3$}
			\put(100,6){$v_4$}
			
			\put(84,26){\color{gray}$\omega_{v_1v_2} = 3/2$}
			\put(78,20){\color{gray}\begin{turn}{-90}$\omega_{v_1v_3} = 1$ \end{turn}}
			\put(83,16){\color{gray} \begin{turn}{-45}$\omega_{v_1v_4} = 1/2$ \end{turn}}
			\put(102,21){\color{gray} \begin{turn}{-90} $\omega_{v_2v_4} = 1/2$ \end{turn}}
			
			\put(34,14){$\Rightarrow$}
			\put(70,14){$\Rightarrow$}
			
			\put(26,53){\begin{minipage}[t][5cm][b]{0.5\textwidth}\begin{align*}\mathcal D = \begin{pmatrix}3&&&\\ &2&&\\ &&1&\\ &&&1\end{pmatrix}\end{align*}\end{minipage}}
			\put(25,37){\begin{minipage}[t][5cm][b]{0.5\textwidth}\begin{align*}\mathcal A = \frac 12\begin{pmatrix}0&3&2&1\\ 3&0&0&1\\ 2&0&0&0\\ 1&1&0&0\end{pmatrix}\end{align*}\end{minipage}}
			
		\end{overpic}
		
		\caption{Visualization of Theorem \ref{thm:effective}. We start with a hypergraph $\Gamma = (V,E)$ with $V=\{v_1,v_2,v_3,v_4\}$ and $E = \{e_1,e_2,e_3\}$. Here, $\mathcal D$ and $\mathcal A$ are chosen according to Example \ref{exam:second} and the resulting weighted graph $\mathcal G$ is depicted on the right.}
		\label{fig:hypergraph_to_graph}
	\end{figure}
	
	As a consequence of Theorem \ref{thm:effective}, we can recover the properties of the random walks on hypergraphs from the theory of random walks on weighted graphs. In the particular case when one can choose the matrix $\mathcal{A}$ to be symmetric, there exists a corresponding underlying graph that is undirected, and Corollary \ref{cor:connected} below holds. Before stating it, we recall the Courant--Fischer--Weyl min-max principle, which allows to characterize the eigenvalues of self-adjoint linear operators.
	
	\begin{theorem}[Courant--Fischer--Weyl min-max principle]\label{thm:minmax}
		Let $W$ be an $n$-dimensional vector space with a positive definite scalar product $(.,.)$. Let $\mathcal{W}_k$ be the family of all $k$-dimensional subspaces of $W$. Let $M : W \rightarrow W$ be a self-adjoint linear operator. Then the eigenvalues $\lambda_1\leq\ldots\leq \lambda_n$ of $M$ can be obtained by
		\begin{equation*}
		\lambda_k=\min_{W_{k}\in\mathcal{W}_{k}}\max_{g(\neq 0)\in W_{k}}\frac{(Mg,g)}{(g,g)}=\max_{W_{n-k+1}\in\mathcal{W}_{n-k+1}}\min_{g(\neq 0)\in W_{n-k+1}}\frac{(Mg,g)}{(g,g)}.
		\end{equation*}
		The vectors $g_k$ realizing such a min-max or max-min then are corresponding eigenvectors, and the min-max spaces $\mathcal{W}_k$ are spanned by the eigenvectors for the eigenvalues $\lambda_1,\ldots,\lambda_k$, and analogously, the max-min spaces $\mathcal{W}_{n-k+1}$ are spanned by the eigenvectors for the eigenvalues $\lambda_k,\ldots,\lambda_n$. Thus, we also have
		\begin{equation}\label{eqminmax}
		\lambda_k=\min_{g\in W,(g,g_j)=0\text{ for }j=1,\ldots,k-1}\frac{(Mg,g)}{(g,g)}=\max_{g\in V,(g,g_l)=0\text{ for }l=k+1,\ldots,n}\frac{(Mg,g)}{(g,g)}.
		\end{equation}
		In particular,
		\begin{equation*}
		\lambda_1=\min_{g\in W}\frac{(Mg,g)}{(g,g)},\qquad \lambda_n=\max_{g\in W}\frac{(Mg,g)}{(g,g)}.
		\end{equation*} 
	\end{theorem}The quantity $\textrm{RQ}_M(g):=\frac{(Qg,g)}{(g,g)}$ in Theorem \ref{thm:minmax} is called the \emph{Rayleigh quotient} of $g$.\newline
	
	We can now state the following:
	
	\begin{corollary}\label{cor:connected}
		Let $\mathcal{L}$ be a random walk Laplacian on a connected hypergraph $\Gamma=(V,E)$. Assume that there exist $\mathcal{A}$ and $\mathcal{D}$ that satisfy \eqref{eq:generalRW}, with $\mathcal{A}$ symmetric and such that
		\begin{equation*}
		\mathcal{A}_{ij}=\mathcal{A}_{ji}\neq 0 \iff i\neq j \text{ and }v_i\sim v_j,
		\end{equation*}where $v_i\sim v_j$ indicates that there exists one hyperedge containing both $v_i$ and $v_j$. Then,
		\begin{enumerate}
			\item The spectrum of $\mathcal{L}$ consists of $N$ real eigenvalues that are contained in $[0,2]$.
			\item $2$ is an eigenvalue for $\mathcal{L}$ if and only if $\Gamma$ is a bipartite graph. In this case, let the vertex set be correspondingly partitioned into $V_1,V_2$, and let $f$ be an eigenfunction for the eigenvalue $2$. Then, $f$ equals some constant $c$ on $V_1$ and $-c$ on $V_2$.
			\item The eigenvalues of $\mathcal{L}$ can be characterized, with the Courant--Fischer--Weyl min-max principle, via the Rayleigh quotients
			\begin{equation*}
			\textrm{RQ}_{\mathcal{L}}(f)=\frac{\sum_{v_i\sim v_j}\mathcal{A}_{ij}\left(f(v_i)-f(v_j)\right)^2}{\sum_{i}\mathcal{D}_{ii}f(v_i)^2},
			\end{equation*}for functions $f:V\rightarrow \mathbb{R}$.
		\end{enumerate}
	\end{corollary}
	\begin{proof}
		It follows from Theorem \ref{thm:effective} and from the spectral properties of the normalized Laplacian for weighted undirected graphs in \cite{Chung,butler2006spectral}.
	\end{proof}
	
	We conclude this section by discussing, in details, the maximum principles for  Laplacians. They hold for a class of Laplacians that include the random walk Laplacians, but for instance also the algebraic Laplacians often employed in graph theory. 
	Let $G$ be an $N\times N$ diagonal matrix with positive diagonal entries, and $A$ be an $N\times N$ non-negative matrix, such that
	\begin{equation}\label{max1}
	G_{ii}=\sum_j A_{ij} \quad \text{ for every } i.
	\end{equation}
	We define the corresponding Laplacian
	as
	\begin{equation*}\label{max2}
	\Lambda = G -A,   
	\end{equation*}
	that is, 
	\begin{equation}\label{max3}
	\Lambda f(v_i)= G_{ii} f(v_i)-\sum_j A_{ij} f(v_j).
	\end{equation}
	This includes the random walk Laplacians of the form \eqref{RWL}. We put $G_{ii}=1$, $A_{ij}=\frac{\mathcal{A}_{ij}}{\mathcal{D}_{ii}}$ and recall \eqref{eq:AD}. It also includes the corresponding non-normalized Laplacian, putting $G_{ii}=\mathcal{D}_{ii}$, $A_{ij}=\mathcal{A}_{ij}$.
	\begin{proposition}\label{prop7.1}
		If
		\begin{equation*}\label{max4}
		\Lambda f_0=0,    
		\end{equation*}
		then $f_0$ is constant on every connected component of $\Gamma$.
	\end{proposition}
	\begin{proof}
		Let $v_i$ be a vertex where $f_0$ achieves its maximum. Then
		\begin{equation*}\label{7.4}
		f_0(v_i)\ge \sum_{v_j \sim v_i}\frac{{A}_{ij}}{G_{ii}}f_0(v_j), 
		\end{equation*}
		and because of \eqref{max1} and the non-negativity of the $A_{ij}$, we can have equality only if $f_0(v_j)=f_0(v_i)$ for every $j$ with ${A}_{ij}>0$. By \eqref{max3}, $f_0$ therefore is also maximal on every vertex $v_j$ connected with $v_i$. By iteration, it is then  constant on the component of $v_i$. This yields the claim, since $f$ has to assume a maximum at some vertex of the finite graph $\Gamma$.
	\end{proof}
	When we look at Dirichlet (or other) boundary value problems, e.g., consider a random walker on a (hyper)graph with boundary, that is, we select some boundary vertex set $V_0\subsetneq V$, with $ V\backslash V_0$ being connected,  the corresponding equation 
	\begin{equation}\label{7.6}
	\Lambda f(v)=0
	\end{equation}   
	would be required only for $v\in V\backslash V_0$. 
	A solution $f$ of \eqref{7.6} is called \emph{harmonic}.
	\begin{corollary}\label{cor7.1}
		When a solution $f$ of \eqref{7.6} achieves a local maximum or minimum at some $v\in  V\backslash V_0$, it is constant. 
	\end{corollary}
	\begin{proof}
		This follows from the proof of Proposition \ref{prop7.1}.
	\end{proof}

	\section{The chemical Laplacian}\label{section:orientedL}
	
	We now recall the normalized Laplacian for oriented hypergraphs introduced in \cite{JM}, that we shall call the \emph{chemical Laplacian}, since it was introduced with the idea of studying chemical reaction networks. In contrast to the random walk and diffusion Laplacians discussed above that relate individual vertices, it relates subsets of vertices, the inputs and the outputs of oriented hyperedges. As introduced in Section \ref{rhyp}, in an oriented hypergraph, the hyperedges have inputs and outputs, but in contrast to directed hypergraphs, it is assumed that both \emph{directions} are present. Hence, the relations between inputs and outputs of an hyperedge are symmetric, and because of that symmetry, the spectrum remains real.
	
	
	Fix an oriented hypergraph $\Gamma=(V,E)$. The chemical Laplacian is based on  the \emph{boundary operator}
	\begin{equation*}
	\delta f(e):=\sum_{v_i \text{ input of }e}f(v_i)-\sum_{v_j \text{ output of }e}f(v_j)
	\end{equation*}
	for a function $f:V\to \R$, as well as on its adjoint operator
	\begin{equation}\label{d2}
	\delta^*(\gamma)(v)=\frac{\sum_{e_{\text{in}}: v\text{ input}}\gamma(e_{\text{in}})-\sum_{e_{\text{out}}: v\text{ output}}\gamma(e_{\text{out}})}{t(v)},
	\end{equation}
	for functions $\gamma$ on oriented hyperedges, where the degree of a vertex $v$ is
	\begin{equation*}
	t(v):=|e\in E\,:\,v\in e|.
	\end{equation*}
	The underlying scalar products
	are
	\begin{equation*}\label{d3}
	(f,g)_V:=\sum_{v\in V}t(v)\cdot f(v)\cdot g(v)
	\end{equation*}
	for  $f,g:V\rightarrow\mathbb{R}$, and 
	\begin{equation*}\label{d4}
	(\omega,\gamma)_E:=\sum_{e\in E}\omega(e)\cdot\gamma(e)
	\end{equation*}
	or $\omega,\gamma:E\rightarrow\mathbb{R}$. The resulting chemical Laplace
	operator then is
	\begin{eqnarray}\nonumber
	L^of(v):=&\delta^*\delta(v)\\
	\nonumber
	=&\frac{\sum_{e_{\text{in}}: v\text{ input}}\biggl(\sum_{v' \text{ input of }e_{\text{in}}}f(v')-\sum_{w' \text{ output of }e_{\text{in}}}f(w')\biggr)}{t(v)}\\
	&-\frac{\sum_{e_{\text{out}}: v\text{ output}}\biggl(\sum_{\hat{v} \text{input of }e_{\text{out}}}f(\hat{v})-\sum_{\hat{w} \text{ output of }e_{\text{out}}}f(\hat{w})\biggr)}{t(v)}.\label{d5}
	\end{eqnarray}
	There is still some degree of freedom in the choice of normalization in \eqref{d2} and therefore also in \eqref{d5}. Thus, one may consider variants with different normalizations, but this does not affect its essential conceptual and mathematical properties.\newline
	The rich dynamics generated by coupling various dynamical systems on hypergraphs via
	this operator have been explored in \cite{Master-Stability} and in \cite{BKJM}. Moreover, a slight generalization of this operator has been offered in \cite{JM21b} and its spectral theory has been recently applied, in \cite{RM-bio}, for studying networks of genetic expression.\newline
	As shown in Theorem \ref{thm:effective}, the random walk Laplacians that we have defined all can be represented as Laplacians on weighted graphs, as the transitions happens between vertices and the probabilities depend on how these vertices are connected. This is no longer true for the  chemical  Laplacian $L^o$ which encodes relations between sets of vertices. It therefore does not possess an interpretation in terms of random walks or diffusion processes between vertices.  While  $L^o$ therefore does not capture any random walk on the hypergraph, it seems to capture some other hypergraph properties in a more precise way than a random walk Laplacian $\mathcal{L}$. For instance, if the assumptions of Corollary \ref{cor:connected} are satisfied,
	\begin{itemize}
		\item The eigenvalues of $\mathcal{L}$ are contained in $[0,2]$, as in the case of graphs, while the eigenvalues of $L^o$ are\,---\,more generally\,---\, contained in $[0,\Psi]$, where $\Psi$ is the largest hyperedge cardinality.
		\item The largest eigenvalue of $\mathcal{L}$ attains its largest possible value $2$ if and only if $\Gamma$ is a bipartite \emph{graph}, while the largest eigenvalue of $L^o$ attains its largest possible value $\Psi$ if and only if $\Gamma$ is $\Psi$-uniform (meaning that all its hyperedges have cardinality $\Psi$) and it is a \emph{bipartite hypergraph}, as shown in \cite{Sharp}.
		\item The eigenvalues of $\mathcal{L}$ are characterized via the Rayleigh quotients
		\begin{equation*}
		\textrm{RQ}_{\mathcal{L}}(f)=\frac{\sum_{v_i\sim v_j}\mathcal{A}_{ij}\left(f(v_i)-f(v_j)\right)^2}{\sum_{i}\mathcal{D}_{ii}f(v_i)^2},
		\end{equation*}while the eigenvalues of $L^o$ are characterized via the Rayleigh quotients
		\begin{equation*}
		\textrm{RQ}_{L^o}(f)=\frac{\sum_{e\in E}\left(\sum_{v_i\in e_{in}}f(v_i)-\sum_{v_j\in e_{out}}f(v_j)\right)^2}{\sum_{i}t(v_i)f(v_i)^2},
		\end{equation*}as shown in \cite{JM}. By looking at these quantities, it becomes clear that the Rayleigh quotients for $L^o$ are better capturing the hypergraph structure.
		\item While we can recover the properties of $\mathcal{L}$ from the theory of random walks on weighted graphs, nothing similar can be done for the chemical  Laplacian $L^o$. We cannot recover the properties of $L^o$ by constructing an auxiliary weighted graph. This is clear, for instance, from the fact that the spectrum of $L^o$ for a hypergraph may be contained in a larger interval than the spectrum of the Laplacian for a graph.
	\end{itemize}
	We also observe that a random walk Laplacian has $0$ as its smallest eigenvalue, and if, as we always assume, the hypergraph is connected, the only eigenfunctions for that eigenvalue are the constants. This follows directly from Theorem \ref{thm:effective} and from the well-known fact that, for a connected graph, the only eigenfunctions of the eigenvalue $0$ are the constants \cite{Chung}. In contrast, as shown in \cite{JM}, there exist connected hypergraphs where $L^o$ does not possess the eigenvalue $0$, and others where $L^o$ has non-constant eigenfunctions for that eigenvalue, thus violating the maximum principle in Proposition \ref{prop7.1}. This shows again, that $L^o$ cannot be interpreted as a random walk Laplacian. More generally, Proposition \ref{prop7.1} tells us that, since  $L^0$  does not  satisfy a maximum principle, it cannot be reduced to a non-normalized, e.g., algebraic graph Laplacian either.\\
	In fact, from \eqref{d5}, we can easily infer the geometric meaning of $L^0$. For an oriented hyperedge, it compares what flows in at the input vertices  of a hyperedge and what flows out at its output vertices  with what flows out at the input vertices and what flows in at the output vertices. In particular, all the input vertices  of a hyperedge contribute in parallel, and so do all its output vertices. Therefore, when the input set or the output set  contains more than one vertex, the maximum principle need not hold. Consider, for instance, a hypergraph with a single oriented hyperedge with $v_1$, $v_2$ as input vertices and $v_3$ as output vertex. Then any $f$ with  $f(v_1)+f(v_2)=f(v_3)$ satisfies $L^0f=0$. Changing the orientation of the hyperedge does not affect this.\medskip
	
	As a next step, it is naturally important to go beyond static properties and investigate the impact that a choice of hypergraph Laplacian can make on dynamics.
	
	\section{Coupled Hypergraph Maps}\label{section:CHM}
	
	In \cite{BKJM}, \emph{Coupled Hypergraph Maps} 
	(CHMs) were defined as a generalization of the well-known Coupled Map Lattices (CMLs)~\cite{Kaneko}, which are discrete-time dynamical systems on graphs. CHMs can be described as follows. Given a connected hypergraph $\Gamma=(V,E)$ on $N$ nodes, we assume that each node $i$ evolves according to a time-discrete map $f:\mathcal{I}\rightarrow\mathcal{I}$, where $\mathcal{I}$ is a compact interval. Two key examples are: the \emph{logistic map} given by
	\begin{equation*}
	g(x):=\mu x(1-x),\qquad \mu\in(0,4],
	\end{equation*}
	and the \emph{tent map} defined by
	\begin{equation*}
	f(x):=\frac{\mu}{2}\min\{x,1-x\},\qquad \mu\in(0,4].
	\end{equation*}
	Note carefully that the parametric assumption $\mu\in(0,4]$ entails that both maps leave the interval $\mathcal{I}=[0,1]$ \emph{invariant} so that they yield well-defined dynamical systems defined in forward discrete time via iteration of the map, e.g., $x_{t+1}=f(x_t)$. Both maps are unimodal with a single global maximum on $[0,1]$. In particular, one can view the tent map as an analytically more tractable variant of the logistic map, so we focus on the tent map in this work, which can already show complicated periodic and chaotic behaviour~\cite{AlligoodSauerYorke}. To represent complex network dynamics, it is natural to study coupled iterated maps~\cite{Kaneko}. Given a vertex $i$ and a time $t\in \mathbb{N}$, we let $x_t(i)\in\mathbb{R}$ denote the state of node $i$ at time $t$, and we assume that $x_t(i)$ is updated via the main dynamical iteration rule
	\begin{equation}
	\label{eq:CHMs}
	x_{t+1}(i)=f(x_t(i))-\varepsilon (\mathcal{O}f)(x_t(i)),
	\end{equation}
	where for classical CMLs one uses the graph Laplacian, while for CHMs $\mathcal{O}$ is a hypergraph Laplacian associated with $\Gamma$, and $\varepsilon$ is a parameter controlling the coupling strength. We can also write \eqref{eq:CHMs} more compactly using the vector notation $\mathbf{x}_t:=(x_t(1),\ldots,x_t(N))$ and by setting
	\begin{equation*}
	C:=C(\mathcal{O})=\id-\varepsilon\mathcal{O},
	\end{equation*} 
	so that
	\begin{equation}
	\label{eq:CHMs2}
	\mathbf{x}_{t+1}=(Cf)(\mathbf{x}_t),
	\end{equation}
	where we now understand the action of $f$ as component-wise extension if it takes a vector, i.e.\ $f(\mathbf{x}_t)=(f(x_t(1)),\ldots,f(x_t(N)))$. Ideally, we would like to analyze the dynamics of~\eqref{eq:CHMs2} by generalizing techniques from the classical theory of one-dimensional iterated maps such as the logistic or tent map~\cite{ColletEckmann2,DeMeloVanStrien,AlligoodSauerYorke}. One important technique are is to characterize invariance properties of intervals. Recall that we know  $f(\mathcal{I})\subseteq \mathcal{I}$. In fact, one can often even make finer statements regarding proper sub-intervals $f(\mathcal{I}_{s_1})\subseteq \mathcal{I}_{s_2}$, where $\mathcal{I}_{s_1}, \mathcal{I}_{s_2} \subset \mathcal{I}$. This forms the basis of symbolic dynamics, i.e., coding the behaviour of the iterated map purely in terms of recording via symbols~\cite{Robinson3}, in which sets a trajectory lands at each time step. A second important technique is renormalization~\cite{Feigenbaum}, which can be used to study the full bifurcation diagram of the tent map upon parameter variation of $\mu$. Renormalization is a more global technique not tracking individual trajectories but it exploits the self-similarity of the map. For example, the $k$-th order iterates of the graph of the tent map correspond (in this case exactly, for the logistic map only locally) to $k$ horizontally compressed and vertically stretched copies of the tent map. This self-similarity can be made precise using a renormalization group (RG) operator on the entire parametric family of maps. For CMLs, it is known that the RG approach can be generalized on a formal level~\cite{LemaitreChate,LemaitreChate1}. Yet, additional problems already appear for CMLs, where one eventually needs commutation properties. More precisely, iterating a shifted graph Laplacian should commute with iterating the map on suitable sub-intervals. To demonstrate that our framework provides some partial steps to analyze the dynamics of CHMs, we want to to check, which properties of the map $f$ are inherited by the operator $C$. In particular, we consider two important natural questions:
	
	\begin{enumerate}
		\item When does it hold that $Cf(\mathbf{x})\in\mathcal{I}^N$, i.e., under which conditions does the dynamics \eqref{eq:CHMs2} leave the $N$-dimensional hypercube $\mathcal{I}^N$ invariant?
		\item Consider the map $C\circ f$ defining the dynamics, say for the tent map. Does it hold true that
		\begin{equation}\label{eq:Q7}
		(C\circ f)^2|_{\mathcal{I}_a}=C^2\circ f^2|_{\mathcal{I}_a},
		\end{equation}for some non-trivial sub-interval $\mathcal{I}_a\subset \mathcal{I}$, i.e.\ do we have some form of commutativity?
	\end{enumerate} 
	
	In order to further develop the mathematical foundations of CHMs, we answer the above questions for the hypergraph Laplacians that have been discussed in the previous sections. 
	
	\begin{theorem}\label{thm:epsilon}
		Let $\mathcal{I}=[a,b]$, with $0\leq a<b$. Let $\mathcal{L}$ be an $N\times N$ random walk Laplacian and let $C:=\id-\varepsilon\mathcal{L}$. Then, given $\mathbf{x}\in \mathcal{I}^N$,
		\begin{equation*}
		Cf(\mathbf{x})\in \mathcal{I}^N \text{ for each }f:\mathcal{I}\rightarrow\mathcal{I}\iff \varepsilon\in [0,1].
		\end{equation*}Moreover, given $f:\mathcal{I}\rightarrow\mathcal{I}$, if $\varepsilon >1$ then
		\begin{equation*} Cf(\mathbf{x})\in [b(1-\varepsilon)+a\varepsilon,a(1-\varepsilon)+b\varepsilon]^N.
		\end{equation*}If $\varepsilon <0$, then
		\begin{equation*} Cf(\mathbf{x})\in [a(1-\varepsilon)+b\varepsilon,b(1-\varepsilon)+a\varepsilon]^N.
		\end{equation*}
	\end{theorem}
	Informally speaking, the first result simply follows from the averaging property of random walk Laplacians.
	\begin{proof}
		Fix two matrices $\mathcal{A}$ and $\mathcal{D}$ that satisfy \eqref{eq:generalRW}, so that $\mathcal{L}=\id-\mathcal{D}^{-1}\mathcal{A}$. Then, by \eqref{eq:AD},
		\begin{equation*}
		\sum_{j}\frac{\mathcal{A}_{ij}}{\mathcal{D}_{ii}}=1,
		\end{equation*}for each $i$. Also, $ C=\id-\varepsilon \id +\varepsilon \mathcal{D}^{-1}\mathcal{A}$, therefore
		\begin{equation*}
		C_{ij}=\begin{cases}
		1-\varepsilon &\text{if }i=j\\
		\frac{\varepsilon \mathcal{A}_{ij}}{\mathcal{D}_{ii}} &\text{if }i\neq j.
		\end{cases}
		\end{equation*}Hence, $Cf(\mathbf{x})\in \mathcal{I}^N$ if and only if
		\begin{equation}\label{eQ1}
		a\leq f(x_i) (1-\varepsilon)+\sum_{j\neq i}f(x_j)\cdot\frac{\varepsilon \mathcal{A}_{ij}}{\mathcal{D}_{ii}}\leq b,
		\end{equation}for each $i=1,\ldots,N$.\newline 
		
		If $\varepsilon \in [0,1]$, then
		\begin{equation*}
		f(x_i) (1-\varepsilon)+\sum_{j\neq i}f(x_j)\cdot\frac{\varepsilon \mathcal{A}_{ij}}{\mathcal{D}_{ii}}\leq \max_{k}f(x_k)\cdot\left(1-\varepsilon+\varepsilon \sum_{j\neq i}\frac{\mathcal{A}_{ij}}{\mathcal{D}_{ii}}\right)=\max_{k}f(x_k)
		\end{equation*}and similarly
		\begin{equation*}
		f(x_i) (1-\varepsilon)+\sum_{j\neq i}f(x_j)\cdot\frac{\varepsilon \mathcal{A}_{ij}}{\mathcal{D}_{ii}}\geq \min_{k}f(x_k)\cdot\left(1-\varepsilon+\varepsilon \sum_{j\neq i}\frac{\mathcal{A}_{ij}}{\mathcal{D}_{ii}}\right)=\min_{k}f(x_k),
		\end{equation*}hence
		\begin{equation*}\label{eq:minmax}
		a\leq \min_{k}f(x_k) \leq  \bigl(C f(\mathbf{x})\bigr)_i \leq \max_{k}f(x_k)\leq b.
		\end{equation*}Therefore \eqref{eQ1} holds true.\newline
		
		If $\varepsilon\notin [0,1]$, assume that $f(x_i)=a$ while $f(x_j)=b$ for all $j\neq i$. Then,
		\begin{equation*}
		\bigl(C f(\mathbf{x})\bigr)_i=a(1-\varepsilon)+b\varepsilon
		\end{equation*}is not in $\mathcal{I}$. In fact,
		\begin{itemize}
			\item If $\varepsilon>1$, then 
			\begin{equation*}
			a(1-\varepsilon)+b\varepsilon>b(1-\varepsilon)+b\varepsilon=b.
			\end{equation*}
			\item If $\varepsilon<0$, then 
			\begin{equation*}
			a(1-\varepsilon)+b\varepsilon<a(1-\varepsilon)+a\varepsilon=a.
			\end{equation*}
		\end{itemize}In both cases,
		\begin{equation*}
		\bigl(C f(\mathbf{x})\bigr)_i=a(1-\varepsilon)+b\varepsilon\notin \mathcal{I}.
		\end{equation*}This proves the first claim.\newline
		
		Now, if $\varepsilon >1$, then $1-\varepsilon< 0$. Therefore, for each $i$,
		\begin{equation*}
		f(x_i) (1-\varepsilon)+\sum_{j\neq i}f(x_j)\cdot\frac{\varepsilon \mathcal{A}_{ij}}{\mathcal{D}_{ii}}\leq a(1-\varepsilon)+b \varepsilon \left(\sum_{j\neq i}\frac{ \mathcal{A}_{ij}}{\mathcal{D}_{ii}}\right)=a(1-\varepsilon)+b\varepsilon,
		\end{equation*}and similarly
		\begin{equation*}
		f(x_i) (1-\varepsilon)+\sum_{j\neq i}f(x_j)\cdot\frac{\varepsilon \mathcal{A}_{ij}}{\mathcal{D}_{ii}}\geq b(1-\varepsilon)+a \varepsilon \left(\sum_{j\neq i}\frac{ \mathcal{A}_{ij}}{\mathcal{D}_{ii}}\right)= b(1-\varepsilon)+a\varepsilon.
		\end{equation*}If $\varepsilon<0$, then $1-\varepsilon>1$. Therefore
		\begin{equation*}
		f(x_i) (1-\varepsilon)+\sum_{j\neq i}f(x_j)\cdot\frac{\varepsilon \mathcal{A}_{ij}}{\mathcal{D}_{ii}}\leq b(1-\varepsilon)+a \varepsilon \left(\sum_{j\neq i}\frac{ \mathcal{A}_{ij}}{\mathcal{D}_{ii}}\right)=b(1-\varepsilon)+a\varepsilon,
		\end{equation*}and
		\begin{equation*}
		f(x_i) (1-\varepsilon)+\sum_{j\neq i}f(x_j)\cdot\frac{\varepsilon \mathcal{A}_{ij}}{\mathcal{D}_{ii}}\geq a(1-\varepsilon)+b\varepsilon \left(\sum_{j\neq i}\frac{ \mathcal{A}_{ij}}{\mathcal{D}_{ii}}\right)=a(1-\varepsilon)+b\varepsilon.
		\end{equation*}Hence, if $\varepsilon>1$,
		\begin{equation*}
		Cf(\mathbf{x})\in [b(1-\varepsilon)+a \varepsilon,a(1-\varepsilon)+b \varepsilon]^N.
		\end{equation*}If $\varepsilon<0$, 
		\begin{equation*}
		Cf(\mathbf{x})\in [a(1-\varepsilon)+b \varepsilon,b(1-\varepsilon)+a \varepsilon]^N.
		\end{equation*}
	\end{proof}
	\begin{remark}
		If, instead of a random walk Laplacian, we consider the chemical  Laplacian $L^o$ that we discussed in Section \ref{section:orientedL}, or its variant in \cite{BKJM}, Theorem \ref{thm:epsilon} does not hold. To see this, we fix an oriented hypergraph $\Gamma=(V,E)$ with nodes $v_1,\ldots,v_N$ and we write $L^o$ in matrix form as
		\begin{equation*}
		L^o=\id - T^{-1}A^o,
		\end{equation*}where $T=(t(v_1),\ldots,t(v_N))$ is the diagonal matrix with entries
		\begin{equation*}
		t(v_i)=|e\in E\, : \, v_i\in e|,
		\end{equation*}and $A^o=(A^o_{ij})$ is the \emph{oriented adjacency matrix} with entries $A^o_{ii}:=0$ for each $i$ and
		\begin{align*}
		A^o_{ij}:=& |\{\text{hyperedges in which }v_i \text{ and }v_j\text{ are anti-oriented}\}|\\
		&-|\{\text{hyperedges in which }v_i \text{ and }v_j\text{ are co-oriented}\}|,
		\end{align*}for $i\neq j$. Similarly, we write the chemical  Laplacian in \cite{BKJM} in matrix form as
		\begin{equation*}
		\Delta^o=D^{-1}T - D^{-1}A^o,
		\end{equation*}
		where $D=(d(v_1),\ldots,d(v_N))$ is the diagonal matrix with entries
		\begin{equation*}
		d(v_i):=\sum_{e\in E\,:\,v_i\in e}(|e|-1).
		\end{equation*}The off-diagonal entries of $L^o$ and $\Delta^o$ are such that
		\begin{equation*}
		L^o_{ij}=\frac{A^o_{ij}}{t(v_i)}\in [-1,1]\quad \text{and }\quad \Delta^o_{ij}=\frac{A^o_{ij}}{d(v_i)}\in [-1,1] \quad \text{for }i\neq j,
		\end{equation*}while their diagonal entries are such that
		\begin{equation*}
		0<\Delta^o_{ii}=\frac{t(v_i)}{d(v_i)}\leq 1=L^o_{ii},
		\end{equation*}and the two operators coincide in the case of graphs. Also, both these Laplacians are not necessarily zero-row sum matrices, in contrast to the case of the random walk Laplacians. Now, as in Theorem \ref{thm:epsilon}, let $\mathcal{I}=[a,b]$, with $0\leq a<b$. Let $f:\mathcal{I}\rightarrow \mathcal{I}$ and, given $\varepsilon\in [0,1]$, consider $C(L^o)=\id-\varepsilon L^o$. Then, given $\mathbf{x}\in \mathcal{I}^N$, $C(L^o)f(\mathbf{x})$ is not necessarily contained in $\mathcal{I}^N$. In fact, observe that \begin{equation*}
		C(L^o)_{ij}=\begin{cases}
		1-\varepsilon &\text{if }i=j\\
		\frac{\varepsilon A^o_{ij}}{t(v_i)} &\text{if }i\neq j.
		\end{cases}
		\end{equation*}Hence, $C(L^o)f(\mathbf{x})\in \mathcal{I}^N$ if and only if, for each $i=1,\ldots,N$,
		\begin{equation}\label{eQ1o}
		a\leq f(x_i) (1-\varepsilon)+\sum_{j\neq i}f(x_j)\cdot\frac{\varepsilon A^o_{ij}}{t(v_i)}\leq b.
		\end{equation}In the particular case where all vertices of $\Gamma$ are inputs for each hyperedge in which they are contained, one has
		\begin{equation*}
		A^o_{ij}=-|\{e\in E\,:\, v_i,v_j\in e\}|, \quad \text{if }i\neq j,
		\end{equation*}therefore
		\begin{equation*}
		\sum_{j}\frac{A^o_{ij}}{t(v_i)}=-\frac{d(v_i)}{t(v_i)}\leq -1.
		\end{equation*}Hence, if there exists $i$ such that $f(x_i)=a$ and $f(x_j)=b$ for all $j$ such that $v_j\sim v_i$, then
		\begin{equation*}
		f(x_i) (1-\varepsilon)+\sum_{j\neq i}f(x_j)\cdot\frac{\varepsilon A^o_{ij}}{t(v_i)}\leq a(1-\varepsilon)-b\varepsilon <a(1-\varepsilon)-a\varepsilon\leq a.
		\end{equation*}
		Therefore, \eqref{eQ1o} does not hold in this case, implying that, in general, $C(L^o)f(\mathbf{x})$ is not necessarily contained in $\mathcal{I}^N$. This can be shown, in a similar way, also for $C(\Delta^o)$.
	\end{remark}
	
	\begin{figure}[!ht]
		
		\hspace{1cm}
		\begin{overpic}[width = .25\linewidth]{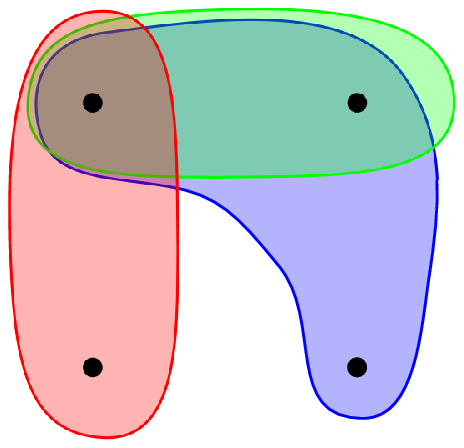}
			\put(-10,-34){\textbf{(a)}}
			
			\put(20,80){$v_1$}
			\put(80,80){$v_2$}
			\put(20,20){$v_3$}
			\put(80,20){$v_4$}
			
			\put(100,45){\color{blue}$e_1$}
			\put(45,100){\color{green}$e_2$}
			\put(-10,45){\color{red}$e_3$}
			
			\put(10,80){\color{blue}$+$}
			\put(23,70){\color{green}$-$}
			\put(10,65){\color{red}$+$}
			
			\put(80,65){\color{blue}$-$}
			\put(70,80){\color{green}$+$}
			
			\put(10,20){\color{red}$-$}
			
			\put(70,20){\color{blue}$+$}
		\end{overpic}
		
		\vspace{-4cm}\hspace{5.5cm}
		\begin{overpic}[width = .7\linewidth]{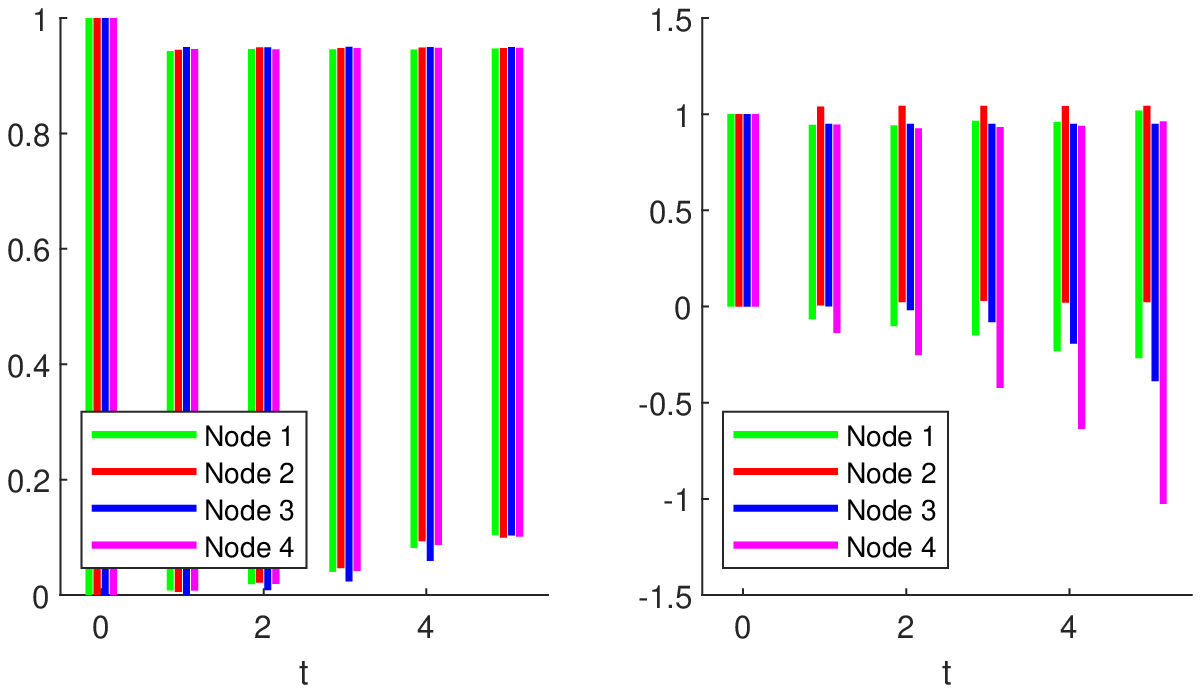}
			\put(5,0){\textbf{(b)}}
			\put(50,0){\textbf{(c)}}
		\end{overpic}
		
		\caption{To illustrate Theorem \ref{thm:epsilon}, we perform numerical simulations of the iterated map \eqref{eq:CHMs2} on an oriented hypergraph $\Gamma = (V,E)$ that is shown in \textbf{(a)}. Here, $V=\{v_1,v_2,v_3,v_4\}$ and $E = \{e_1,e_2,e_3\}$. A node $v_i$ is classified as an input (output) of an hyperedge $e_j$ if there is a plus (minus) sign in the respective color of $e_j$ next to $v_i$. We choose $\mathcal I = [0,1]$, $f$ to be the tent map with $\mu = 3.8$ and $\varepsilon = 0.3$. \textbf{(b)} shows simulation results with $\mathcal L$ being a random walk Laplacian for $10^6$ different uniformly distributed initial conditions in $\mathcal I^4$. In fact, the green, red, blue and magenta bars at $t=0$ in \textbf{(b)} depict the marginals of this uniform distribution projected to the first, second, third and fourth component, respectively. For $t>0$ all of these bars are still contained in $\mathcal I$, which illustrates Theorem \ref{thm:epsilon}. \textbf{(c)} shows the exact same simulation except that here $\mathcal L$ is not a random walk Laplacian but instead given by $L^o$ from \eqref{d5}. \textbf{(b)-(c)} show the clear dynamical difference of using different hypergraph Laplacians for CHMs.}
		\label{fig:interval_invariance}
	\end{figure}
	
	Before stating the next theorem pertaining to our second question, we make an observation regarding Eq.\ \eqref{eq:Q7}.
	\begin{remark}\label{rmk:Q7}
		Let $\mathcal{I}=[a,b]$, with $0\leq a<b$. Let $f:\mathcal{I}\rightarrow \mathcal{I}$ and let $C=\id-\varepsilon\mathcal{O}$, for some $N\times N$ Laplacian matrix $\mathcal{O}$. Let also $\mathbf{x}=(x_1,\ldots,x_N)\in \mathcal{I}_a^N$, so that $f(\mathbf{x})=(f(x_1),\ldots,f(x_N))\in \mathcal{I}^N$. Then,
		\begin{align*}
		(C\circ f)^2(\mathbf{x})&=(C\circ f)\biggl(\left(C f(\mathbf{x})\right)_i\biggr)_{i=1,\ldots,N}\\
		&=(C\circ f)\left(f(x_i) (1-\varepsilon \mathcal{O}_{ii})-\sum_{j\neq i}f(x_j)\cdot\varepsilon \mathcal{O}_{ij}\right)_{i=1,\ldots,N}\\
		&=C\left(f\biggl(f(x_i) (1-\varepsilon \mathcal{O}_{ii})-\sum_{j\neq i}f(x_j)\cdot\varepsilon \mathcal{O}_{ij}\biggr)\right)_{i=1,\ldots,N}.
		\end{align*}Hence, for $k=1,\ldots,N$,
		\begin{equation}\label{Q7A}
		\biggr( (C\circ f)^2(\mathbf{x})\biggl)_k=f(z_k) (1-\varepsilon \mathcal{O}_{kk})-\sum_{l\neq k}f(z_l)\cdot\varepsilon \mathcal{O}_{kl},
		\end{equation}where
		\begin{equation*}
		z_i:=f(x_i) (1-\varepsilon \mathcal{O}_{ii})-\sum_{j\neq i}f(x_j)\cdot\varepsilon \mathcal{O}_{ij}.
		\end{equation*}On the other hand,
		\begin{align*}
		C^2\circ f^2(\mathbf{x})&=C\biggl(C(f^2(\mathbf{x}))\biggr)\\
		&=C\biggl(\biggl(\bigl(C\cdot f^2(\mathbf{x})\bigr)_i\biggr)_{i=1,\ldots,N}\biggr)\\
		&=C\biggl(\biggl(f^2(x_i) (1-\varepsilon \mathcal{O}_{ii})-\sum_{j\neq i}f^2(x_j)\cdot\varepsilon \mathcal{O}_{ij}\biggr)_{i=1,\ldots,N}\biggr).
		\end{align*}Hence, for $k=1,\ldots,N$,
		\begin{equation}\label{Q7B}
		\biggr( C^2\circ f^2(\mathbf{x})\biggl)_k=y_k (1-\varepsilon \mathcal{O}_{ii})-\sum_{l\neq k}y_l\cdot\varepsilon \mathcal{O}_{kl},
		\end{equation}where
		\begin{equation*}
		y_i:=f^2(x_i) (1-\varepsilon \mathcal{O}_{ii})-\sum_{j\neq i}f^2(x_j)\cdot\varepsilon \mathcal{O}_{ij}.
		\end{equation*}Hence, \eqref{eq:Q7} holds if and only if \eqref{Q7A} is equal to \eqref{Q7B}, for all $k=1,\ldots,N$.\newline 
		This is true if, for all $i=1,\ldots,N$,
		\begin{equation}\label{eq:linear}
		f\Biggl(f(x_i) (1-\varepsilon \mathcal{O}_{ii})-\sum_{j\neq i}f(x_j)\cdot\varepsilon \mathcal{O}_{ij}\Biggr)=f^2(x_i) (1-\varepsilon \mathcal{O}_{ii})-\sum_{j\neq i}f^2(x_j)\cdot\varepsilon \mathcal{O}_{ij}.
		\end{equation}
	\end{remark}
	
	\begin{theorem}
		Let $\mathcal{I}=[0,1]$ and let $f:\mathcal{I}\rightarrow\mathcal{I}$ be a tent map, for some parameter $\mu\in [0,4]$. Let $\mathcal{L}$ be an $N\times N$ random walk Laplacian and let $C:=\id-\varepsilon\mathcal{L}$. Then, for $\mathcal{I}_a=[0,\frac{1}{4}]$, we have that
		\begin{equation*}
		(C\circ f)^2|_{\mathcal{I}_a}=C^2\circ f^2|_{\mathcal{I}_a}.
		\end{equation*}
	\end{theorem}
	\begin{proof}Fix two matrices $\mathcal{A}$ and $\mathcal{D}$ that satisfy \eqref{eq:generalRW}, so that $\mathcal{L}=\id-\mathcal{D}^{-1}\mathcal{A}$.\newline Observe that, for each $x\in \mathcal{I}_a$, $x\leq \frac{1}{4}\leq \frac{1}{\mu}$, therefore
		\begin{equation*}
		f(x)=\frac{\mu}{2}\cdot x\leq \frac{1}{2}
		\end{equation*}and hence
		\begin{equation*}
		f^2(x)=\frac{\mu}{2}\cdot f(x)=\frac{\mu^2}{4}\cdot x.
		\end{equation*}Therefore, the right-hand side of \eqref{eq:linear} in this case can be re-written as  
		\begin{equation*}
		f^2(x_i) (1-\varepsilon)+\sum_{j\neq i}f^2(x_j)\cdot\frac{\varepsilon \mathcal{A}_{ij}}{\mathcal{D}_{ii}}=\frac{\mu^2}{4}\Biggl(x_i (1-\varepsilon)+\sum_{j\neq i}x_j\cdot\frac{\varepsilon \mathcal{A}_{ij}}{\mathcal{D}_{ii}}\Biggr),
		\end{equation*}while the left-hand side of \eqref{eq:linear} is
		\begin{align*}
		f\Biggl(f(x_i) (1-\varepsilon)+\sum_{j\neq i}f(x_j)\cdot\frac{\varepsilon \mathcal{A}_{ij}}{\mathcal{D}_{ii}}\Biggr)&=f\Biggl(\frac{\mu}{2}\cdot\biggl( x_i (1-\varepsilon)+\sum_{j\neq i} x_i\cdot\frac{\varepsilon \mathcal{A}_{ij}}{\mathcal{D}_{ii}}\biggr)\Biggr)\\
		&=\frac{\mu}{2}\cdot \Biggl(\frac{\mu}{2}\cdot\biggl( x_i (1-\varepsilon)+\sum_{j\neq i} x_i\cdot\frac{\varepsilon \mathcal{A}_{ij}}{\mathcal{D}_{ii}}\biggr)\Biggr),
		\end{align*}where the last equality follows from the fact that
		\begin{equation*}
		\frac{\mu}{2}\cdot\biggl( x_i (1-\varepsilon)+\sum_{j\neq i} x_i\cdot\frac{\varepsilon \mathcal{A}_{ij}}{\mathcal{D}_{ii}}\biggr)\leq \frac{\mu}{2}\cdot \max_{k}x_k\leq \frac{\mu}{2}\cdot\frac{1}{\mu}=\frac{1}{2}.
		\end{equation*}Hence, \eqref{eq:linear} holds. By Remark \ref{rmk:Q7}, the claim follows.
	\end{proof}

	\textbf{Acknowledgments:} RM was supported by The Alan Turing Institute under the EPSRC grant EP/N510129/1. CK was supported a Lichtenberg Professorship of the VolkswagenStiftung. TB thanks the TUM Institute for Advanced Study (TUM-IAS) for support through a Hans Fischer Fellowship awarded to Chris Bick. TB also acknowledges support of the TUM TopMath elite study program. JJ is supported by GIF Research Grant No. I-1514-304.6/2019.
	
	\bibliography{RWs}
	
\end{document}